\documentclass[final,leqno]{siamltex704}
\usepackage{amsmath}
\usepackage{graphicx}
\usepackage[notcite,notref]{showkeys}
\usepackage{mathrsfs}
\usepackage{float}
\usepackage{amsfonts,amssymb}
\usepackage{dsfont}
\usepackage{pifont}
\usepackage{wrapfig} % Allows in-line images
\usepackage{multirow}
\numberwithin{equation}{section}
\def\3bar{{|\hspace{-.02in}|\hspace{-.02in}|}}
\def\E{{\mathcal{E}}}
\def\T{{\mathcal{T}}}

\def\bn{{\mathbf{n}}}

\newtheorem{defi}{Definition}[section]

\newtheorem{remark}{Remark}[section]
\newtheorem{algorithm}{Weak Galerkin Algorithm}
\def\ad#1{\begin{aligned}#1\end{aligned}}   
 \def\an#1{\begin{align}#1\end{align}} 
 
\def\p#1{\begin{pmatrix}#1\end{pmatrix}}

\title {A Simple Weak Galerkin Finite Element Method for a Class of Fourth-Order Problems in Fluorescence Tomography}

\author{
 Chunmei Wang \thanks{Department of Mathematics, University of Florida, Gainesville, FL 32611, USA (chunmei.wang@ufl.edu). The research of Chunmei Wang was partially supported by National Science Foundation Grant DMS-2136380.}
  \and
 Shangyou Zhang\thanks{Department of Mathematical Sciences,  University of Delaware, Newark, DE 19716, USA (szhang@udel.edu).  }
 }

\begin{document}

\maketitle

\begin{abstract}
In this paper, we propose a simple numerical algorithm based on the weak Galerkin (WG) finite element method for a class of fourth-order problems in fluorescence tomography (FT), eliminating the need for stabilizer terms required in traditional WG methods. FT is an emerging, non-invasive 3D imaging technique that reconstructs images of fluorophore-tagged molecule distributions in vivo. By leveraging bubble functions as a key analytical tool, our method extends to both convex and non-convex elements in finite element partitions, representing a significant advancement over existing stabilizer-free WG methods. It overcomes the restrictive conditions of previous approaches, offering substantial advantages. The proposed method preserves a simple, symmetric, and positive definite structure. These advantages are confirmed by optimal-order error estimates in a discrete $H^2$ norm, demonstrating the effectiveness and accuracy of our approach. Numerical experiments further validate the efficiency and precision of the proposed method.

\end{abstract}



\pagestyle{myheadings}

\section{Introduction}
This paper focuses on the development of numerical methods for a class of fourth-order problems with Dirichlet and Neumann boundary conditions. The model problem seeks an unknown function  $u=u(x)$ satisfying
\begin{equation}\label{0.1}
\begin{split}
( -\nabla\cdot(\kappa \nabla)+\mu)^2u&=f, \quad\text{in}\ \Omega,\\
u&=\xi, \quad\text{on}\ \partial\Omega,\\
\kappa  \nabla u\cdot \textbf{n} &=\nu, \quad  \text{on}\ \partial\Omega,\\
\end{split}
\end{equation}
where   $\Omega$ is an open, bounded domain in  $\mathbb{R}^d$ ($d=2,3$)
with a Lipschitz continuous boundary $\partial\Omega$. Here, $\bn$ denotes the unit outward normal to $\partial\Omega$, $\kappa$ is a  
 symmetric and positive definite matrix-valued function, and $\mu$ is a nonnegative real-valued function.
The functions 
$f$, $\xi$, and $\nu$ are given in the domain or on its boundary, as appropriate.

 For convenience, we denote the second-order elliptic operator $ \nabla\cdot(\kappa \nabla)$ by $E$. To simplify the analysis, and without loss of generality, we assume that 
 $\kappa$ is piecewise constant matrix and $\mu$  is a piecewise constant function.

The fourth order model problem (\ref{0.1}) is derived from fluorescence tomography (FT)  \cite{ yinkepaper, gkr2014, gz2013, ltka2015, mtspa2011, wbwm2003, wbwm2004,  yinkethesis}, an advanced noninvasive 3D imaging technique used for in vivo applications. 
Unlike traditional imaging methods that rely on X-rays or strong magnetic fields, FT utilizes highly specific fluorescent probes and non-ionizing near-infrared radiation \cite{p-25}, reducing potential health risks. The primary objective of FT is to reconstruct the spatial distribution of fluorophores, which are bound to target molecules, based on boundary measurements. Given its capability to provide molecular-level imaging, FT has emerged as a valuable tool for early cancer detection and drug monitoring \cite{p-5, p-24, p-39}. 

We define the function space
$$
H_{\kappa}^2(\Omega)=\{v:v\in H^1(\Omega),\kappa \nabla v\in H(\text{div};\Omega)\},
$$
equipped with the norm
$$
\|v\|_{\kappa,2}=(\|v\|_1^2+\|\nabla\cdot (\kappa \nabla v)\|^2)^{\frac{1}{2}}.
$$

A variational formulation for the fourth-order model problem \eqref{0.1} seeks a function 
 $u\in H_{\kappa}^2(\Omega)$ satisfying the boundary conditions $u|_{\partial
\Omega}=\xi$, $\kappa  \nabla u\cdot \textbf{n}|_{\partial
\Omega}  =\nu$, such that
\begin{equation}\label{0.2}
 (Eu,Ev)+2\mu (\kappa \nabla u,\nabla v)+\mu^2(u,v)=(f,v), \quad
\forall v\in {\cal V},
\end{equation}
where  $(\cdot,\cdot)$ denotes the standard inner product in
$L^2(\Omega)$. The test space ${\cal V}$  is given by 
$$
{\cal V}=\{v\in H_\kappa^2(\Omega):v|_{\partial\Omega}=0,\kappa \nabla v\cdot \textbf{n}|_{\partial \Omega}=0\}.
$$
Here, $\bn$ represents the outward unit normal to the boundary of $\Omega$.

 The weak Galerkin (WG) finite element method represents a significant advancement in numerical techniques for solving partial differential equations (PDEs). It reconstructs or approximates differential operators in a manner analogous to the theory of distributions for piecewise polynomials. Unlike traditional finite element methods, WG relaxes the usual regularity requirements on approximating functions by incorporating carefully designed stabilizers. Extensive research has demonstrated the effectiveness of WG across a variety of model PDEs, as evidenced by a comprehensive set of references \cite{wg1, wg2, wg3, wg4, wg5, wg6, wg7, wg8, wg9, wg10, wg11, wg12, wg13, wg14, wg15, wangft, wg17, wg18, wg19, wg20, wg21, itera, wy3655}, underscoring its potential as a powerful tool in scientific computing.

A key distinguishing feature of WG methods is their reliance on weak derivatives and weak continuities to construct numerical schemes based on the weak formulation of PDEs. This structural flexibility enhances their applicability across a broad spectrum of PDEs, ensuring both stability and accuracy in numerical approximations.

An important advancement within the WG framework is the Primal-Dual Weak Galerkin (PDWG) method, which addresses challenges that conventional numerical techniques often encounter \cite{pdwg1, pdwg2, pdwg3, pdwg4, pdwg5, pdwg6, pdwg7, pdwg8, pdwg9, pdwg10, pdwg11, pdwg12, pdwg13, pdwg14, pdwg15}. PDWG formulates numerical solutions as constrained minimizations of functionals, where the constraints capture the weak formulation of PDEs through weak derivatives. This leads to an Euler-Lagrange equation that integrates both the primal variable and a dual variable (Lagrange multiplier), resulting in a symmetric numerical scheme.

The variational formulation in \eqref{0.2} differs significantly from the standard biharmonic equation. First, conventional $H^2$-conforming finite elements designed for the biharmonic problem are not necessarily $H_\kappa^2$-conforming and therefore cannot be directly applied to \eqref{0.2}. Additionally, widely used nonconforming elements, such as the Morley element \cite{m1968}, typically rely on  formulations involving the full Hessian. Since it is unclear whether \eqref{0.2} can be reformulated to align with such elements, their direct applicability remains uncertain. In fact, we suspect that these elements may not be suitable for this problem.

This paper introduces a simplified weak Galerkin (WG) finite element formulation that eliminates the need for stabilizers required in \cite{wangft}, significantly streamlining both the numerical scheme and its implementation. Furthermore, our approach extends the applicability of WG methods to both convex and non-convex polytopal meshes, whereas existing methods have been limited to convex cases. A key analytical tool enabling these advancements is the use of bubble functions. Our method preserves the size and global sparsity of the stiffness matrix, reducing programming complexity compared to traditional stabilizer-dependent WG methods. Theoretical analysis establishes that our WG approximations achieve optimal error estimates in the discrete $H^2$ norm. By introducing a stabilizer-free WG method that maintains accuracy while improving computational efficiency, this paper makes a significant contribution to the development of finite element methods for a class of fourth-order problems on non-convex polytopal meshes.

The paper is structured as follows. Section 2 provides a concise review of weak differential operators and their discrete counterparts. In Section 3, we introduce a simple
weak Galerkin algorithm for the fourth-order model problem (\ref{0.1}), formulated based on the variational approach in (\ref{0.2}). Section 4 presents local $L^2$ projection operators and establishes key approximation properties essential for the convergence analysis. Section 5 is dedicated to deriving the error equation for the WG finite element solution. In Section 6, we prove an optimal-order error estimate for the WG finite element approximation in an $H^2$-equivalent discrete norm. Finally, Section 7 provides numerical results that validate the theoretical findings presented in the preceding sections.

\section{Weak Differential Operators and Their Discrete Counterparts} 
In this section, we briefly review the definitions of the weak second-order elliptic operator 
$E$ \cite{wangft} and the weak gradient operator \cite{wy3655}, along with their discrete formulations.

Let ${\cal T}_h$ be a partition of the domain $\Omega$ into polygons
in  2D or polyhedra in 3D.  We assume that ${\cal T}_h$ satisfies the shape regularity condition as defined in \cite{wy3655}.
Denote by $\E_h$ the set of all edges (in 2D) or flat faces (in 3D) within ${\cal T}_h$, and let
$\E_h^0=\E_h\setminus\partial\Omega$ represent the set of all interior edges or faces.   A weak function on an element $T\in {\cal T}_h$ is defined as a triplet $v=\{v_0,v_b,v_g\}$, where $v_0\in L^2(T)$ represents the function's value in the interior of 
$T$,
$v_b\in L^{2}(\partial T)$ represents the function's value on the boundary of $T$, and $v_g\in L^{2}(\partial
T)$ approximates the normal flux 
$\kappa\nabla v \cdot \textbf{n}$ on  $\partial T$, with $\textbf{n}$ denoting the outward unit normal.
 On each interior edge or face 
$e\in {\cal E}_h^0$, shared by two elements $T_L$ and $T_R$, the function $v_g$ has two values: $v_g^L$, as seen from element $T_L$, and $v_g^R$, as seen from element $T_R$ which satisfy the condition $v_g^L + v_g^R = 0$.

The space of all weak functions on 
$T$ is given by 
$$
W(T)=\{v=\{v_0,v_b,v_g\}: v_0\in L^2(T), v_b\in
L^{2}(\partial T), v_g\in L^{2}(\partial T) \}.
$$ 

\begin{defi}\cite{wangft}  For any $v\in
W(T)$, the weak second-order elliptic operator, denoted as 
$E_{ w} v$, is defined as a linear functional in the dual space of $H^2(T)$.  Its action on 
each $\varphi \in H^2(T)$ is given by
 \begin{equation}\label{2.3}
 (E_{ w}v,\varphi)_T=(v_0,E\varphi)_T-
 \langle v_b ,\kappa \nabla \varphi\cdot  \textbf{n}\rangle_{\partial T}+
 \langle v_{g },\varphi  \rangle_{\partial T}.
 \end{equation}
\end{defi}

For any
non-negative integer $r\geq 0$, let $P_r(T)$ denote the space of polynomials on $T$ with a degree at most $r$. The discrete weak second-order elliptic operator, denoted by 
$E_{ w,r,T}$, is defined as the unique polynomial
 $E_{ w,r,T} v\in P_r(T)$ satisfying 
  \begin{equation}\label{2.4}
  (E_{ w}v,\varphi)_T=(v_0,E\varphi)_T-
 \langle v_b ,\kappa \nabla \varphi\cdot  \textbf{n}\rangle_{\partial T}+
 \langle v_{g },\varphi  \rangle_{\partial T},\quad \forall \varphi \in
 P_r(T). 
 \end{equation}
 Using integration by parts, this can be rewritten as 
 \begin{equation}\label{A.002}
(E_{w}v,
\varphi)_T=(Ev_0,\varphi)_T+\langle
v_0-v_b, \kappa\nabla\varphi\cdot \textbf{n} \rangle_{\partial T}-\langle
 \kappa \nabla v_0\cdot \textbf{n}-v_{g },\varphi  \rangle_{\partial T},\forall \varphi \in
 P_r(T).
\end{equation}

\begin{defi} \cite{wy3655} The weak gradient operator for any  $v\in
W(T)$, denoted  as $ \nabla_{ w} v$,   is defined as a linear vector functional in the dual space of $[H^1(T)]^d$.  Its action on each $\boldsymbol{ \psi} \in [H^1(T)]^d$ is given by
 \begin{equation}\label{2.3-2}
 ( \nabla_{ w}v,\boldsymbol{ \psi})_T=-(v_0, \nabla \cdot \boldsymbol{ \psi})_T+
 \langle v_b ,\boldsymbol{ \psi}\cdot  \textbf{n}\rangle_{\partial T}.
 \end{equation}
\end{defi}

 The discrete weak gradient operator, denoted by
$ \nabla_{ w,r,T}$, is defined as the unique vector polynomial 
 $ \nabla_{ w,r,T} v\in [P_r(T)]^d$ satisfying 
  \begin{equation}\label{2.4-2}
  ( \nabla_{ w}v,\boldsymbol{ \psi})_T=-(v_0, \nabla \cdot \boldsymbol{ \psi})_T+
 \langle v_b ,\boldsymbol{ \psi}\cdot  \textbf{n}\rangle_{\partial T},\quad \forall \boldsymbol{ \psi}\in
 [P_r(T)]^d. 
 \end{equation}
 Using integration by parts, this can be reformulated as
 \begin{equation}\label{2.4-3}
  ( \nabla_{ w}v,\boldsymbol{ \psi})_T= (\nabla v_0,  \boldsymbol{ \psi})_T-
 \langle v_0-v_b ,\boldsymbol{ \psi}\cdot  \textbf{n}\rangle_{\partial T},\quad \forall \boldsymbol{ \psi}\in
 [P_r(T)]^d. 
 \end{equation}
 
\section{Stabilizer-Free Weak Galerkin Methods}\label{Section:WGFEM}
For any given integer $k\geq 1$, let $W_k(T)$ denote the local discrete weak function space defined as
\an{\label{Wk}
W_k(T) & =\big\{v=\{v_0,v_b, v_g\}: v_0\in P_k(T), v_b\in
P_{k }(e),v_g\in  P_{k-1}(e), e\subset \partial
T\big\}. }

By assembling $W_k(T)$ over all the elements $T\in {\cal T}_h$ and enforcing continuity across the interior edges $\E_h^0$, we define the global weak finite element space $V_h$  as  
$$
V_h=\big\{v=\{v_0,v_b,v_g\}:\{v_0,v_b,v_g\}|_T\in
W_k(T), \forall T\in {\cal T}_h\big\}.
$$

Furthermore, we define the subspace $V_h^0$ consisting of functions with vanishing traces,
$$
V_h^0=\{v=\{v_0,v_b,v_g\}\in
V_h,v_b|_e=0,v_g|_e=0, e\subset \partial T\cap
\partial\Omega\}.
$$

Denote by
$E_{w, r_1}$ and  $ \nabla_{w, r_2}$ the discrete weak second-order elliptic operator and the discrete weak gradient operator, respectively, which are computed on each element $T$ for $k\geq 1$ using equations (\ref{2.4}) and (\ref{2.4-2}). Specifically, for $v\in V_h$, we have
$$
(E_{ w, r_1} v)|_T=E_{ w, r_1,T}(v|_T),  
 \qquad 
( \nabla_{ w, r_2} v)|_T= \nabla_{ w, r_2,T}(v|_T).
$$

For simplicity, we omit the subscripts $r_1$ and $r_2$ in $E_{ w, r_1}$ and $ \nabla_{ w, r_2} $.

For each element $T$, let $Q_0$ be the $L^2$ projection onto $P_{k}(T)$, and for each edge or face $e\subset\partial T$,
let $Q_{b}$ and $Q_{g}$ be the $L^2$ projections onto $P_{k}(e)$ and $P_{k-1}(e)$, respectively. Given $u\in H^2(\Omega)$,  we define a projection onto the weak finite element space $V_h$ by
$$
Q_hu=\{Q_0u,Q_bu,Q_g(\kappa \nabla u \cdot \textbf{n})\}.
$$ 

We now present the stabilizer-free weak Galerkin finite element scheme for the fourth-order model problem (\ref{0.1}), based on the variational formulation (\ref{0.2}).

\begin{algorithm} Find $u_h=\{u_0,u_b, u_g\}\in V_h$
satisfying the boundary conditions 
 $u_b=Q_{b}\xi$ and $u_g =Q_{g}\nu$  on
$\partial\Omega$, such that
\begin{equation}\label{PDWG1}
\sum_{T\in {\cal T}_h}(E_{w} u_h,E_{w}v )_T +2\mu (\kappa \nabla_w u_h,\nabla_w v )_T +\mu^2(u_0,v_0)_T  =\sum_{T\in {\cal T}_h}(f,v_0)_T,   
\forall v \in V_h^0.
\end{equation}
\end{algorithm}

 For any $v\in V_h$, define an energy norm $\3barv\3bar$ by
\begin{equation}\label{3barnorm}
\3barv\3bar^2= \sum_{T\in {\cal T}_h}(E_{w} v,E_{w}v )_T+2\mu (\kappa \nabla_w v,\nabla_w v )_T +\mu^2(v_0,v_0 )_T .
\end{equation}
Additionally, we define the discrete $H^2$ semi-norm as 
 \begin{equation}\label{disnorm}
     \begin{split}
 & \|v\|_{2, h}^2= \sum_{T\in {\cal T}_h}(E v_0,E v_0)_T+2\mu (\kappa \nabla  v_0,\nabla  v_0)_T+\mu^2(v_0, v_0)_T\\&+   h_T^{-1}\langle  \kappa \nabla
u_0 \cdot \textbf{n}- u_g, \kappa \nabla
v_0 \cdot \textbf{n}-
v_g\rangle_{\partial T}+ h_T^{-3}\langle  u_0-u_b,  
v_0-v_b\rangle_{\partial T}.       
     \end{split}
 \end{equation}

For any element $T\in\T_h$ and any function $\varphi\in H^1(T)$, the trace inequality (see \cite{wy3655}) asserts that
\begin{equation}\label{trace-inequality}
\|\varphi\|_{\partial T}^2 \leq C
(h_T^{-1}\|\varphi\|_T^2+h_T\|\nabla\varphi\|_T^2).
\end{equation}
Moreover, if  $\varphi$ is a polynomial on $T$, then we
have from the inverse inequality (see also \cite{wy3655}) that
\begin{equation}\label{x}
\|\varphi\|_e^2 \leq C h_T^{-1}\|\varphi\|_T^2.
\end{equation}
Here $e$ is an edge or flat face on the boundary of $T$.

\begin{lemma}\label{norm1}
 For $v=\{v_0, v_b, v_g\}\in V_h$, there exists a constant $C$ such that
 $$
 \|E v_0\|_T\leq C\|E_{w} v\|_T.
 $$
\end{lemma}
\begin{proof} Let  $T\in {\cal T}_h$ be a polytopal element with $N$ edges or faces, denoted by $e_1, \cdots, e_N$. Note that  $T$ may be non-convex. For each edge/face $e_i$, we define a linear function $l_i(x)$  satisfying  $l_i(x)=0$ on $e_i$ as follows: 
$$l_i(x)=\frac{1}{h_T}\overrightarrow{AX}\cdot \bn_i, $$  where  $A$ is a given point on  $e_i$,  $X$ is any point on  $e_i$, $\bn_i$ is the unit normal to $e_i$, and $h_T$ represents the  size of $T$. 

The bubble function associated with  $T$ is defined as 
 $$
 \Phi_B =l^2_1(x)l^2_2(x)\cdots l^2_N(x) \in P_{2N}(T).
 $$ 
 It is straightforward to verify that  $\Phi_B=0$ on  $\partial T$. Moreover, 
  $\Phi_B$ can be normalized so that $\Phi_B(M)=1$ where   $M$  denotes the barycenter of $T$. Additionally,  there exists a sub-domain $\hat{T}\subset T$ where $\Phi_B\geq \rho_0$ for some constant $\rho_0>0$.

For $v=\{v_0, v_b, v_g\}\in V_h$, setting $r=2N+k-2$ and  choosing  $\varphi=\Phi_B E v_0\in P_r(T)$ in \eqref{A.002} leads to 
 \begin{equation} \label{t1}
  \begin{split}
  &(E_{w}v, \Phi_B E v_0)_T\\=&(Ev_0,\Phi_B E v_0)_T+\langle
v_0-v_b, \kappa\nabla (\Phi_B E v_0) \cdot \textbf{n} \rangle_{\partial T}-\langle
 \kappa \nabla v_0\cdot \textbf{n}-v_{g }, \Phi_B E v_0  \rangle_{\partial T}\\
 =&(Ev_0,\Phi_B E v_0)_T,
  \end{split}
  \end{equation}
where the boundary condition 
$\Phi_B=0$ on $\partial T$  is used.

Using the domain inverse inequality \cite{wy3655}, there exists a constant $C$ such that 
\begin{equation}\label{t2}
(E v_0, \Phi_B E v_0)_T \geq C (E v_0, E v_0)_T.
\end{equation} 

Applying the Cauchy–Schwarz inequality together with \eqref{t1}–\eqref{t2}, we obtain
 $$
 (E v_0, E v_0)_T\leq C (E_{w} v, \Phi_B E v_0)_T  \leq C  \|E_{w} v\|_T \|\Phi_B E v_0\|_T  \leq C
\|E_{w} v\|_T \|E v_0\|_T.
 $$
Thus, we conclude
 $$
 \|E v_0\|_T\leq C\|E_{w} v\|_T.
 $$

This completes the proof.
\end{proof}

\begin{remark}
  When the polytopal element  $T$  is convex, 
   the bubble function introduced in Lemma \ref{norm1} can be   simplified to
 $$
 \Phi_B =l_1(x)l_2(x)\cdots l_N(x).
 $$  
This function satisfies the following properties: (1) $\Phi_B=0$
 on $\partial T$, (2) there exists a  subregion $\hat{T}\subset T$ where $\Phi_B\geq \rho_0$ for some constant $\rho_0>0$.
 
Using this simplified bubble function, Lemma \ref{norm1} can be established through the same reasoning as before, with the parameter choice  $r=N+k-2$.  
\end{remark}

\begin{lemma}\cite{wangauto}\label{norm1-1}
For any function $v=\{v_0, v_b, v_g\}\in V_h$, there exists a constant $C$ such that
 $$
 \|\nabla v_0\|_T\leq C\|\nabla_{w} v\|_T.
 $$
\end{lemma}

To define an edge/face-based bubble function, we introduce
$$\varphi_{e_k}= \Pi_{i=1, \cdots, N, i\neq k}l_i^2(x).$$   This function satisfies: (1) $\varphi_{e_k}=0$ on every edge/face  $e_i$ for $i \neq k$, (2) there exists a subset $\widehat{e_k}\subset e_k$ such that  $\varphi_{e_k} \geq \rho_1$ for some constant $\rho_1>0$.  Now, let $$\varphi=(v_b-v_0)l_k \varphi_{e_k}.$$ Then, it follows that $\varphi=0$ on all edges/faces $e_i$ for $i=1, \cdots, N$, and $\nabla \varphi =0$ on $e_i$ for $i \neq k$.  Moreover,  the gradient takes the form $\nabla \varphi =(v_0-v_b)(\nabla l_k) \varphi_{e_k}=\mathcal{O}( \frac{ (v_0-v_b)\varphi_{e_k}}{h_T}\textbf{C})$  on $e_k$ for some vector constant $\textbf{C}$.
\begin{lemma}\label{phi}
     For $\{v_0,v_b, v_g\}\in V_h$, let $\varphi=(v_b-v_0)l_k \varphi_{e_k}$. Then, the inequality
\begin{equation}
  \|\varphi\|_T ^2 \leq Ch_T  \int_{e_k}(v_b-v_0)^2ds.
\end{equation}
holds.\end{lemma}
\begin{proof}
To extend  $v_b$, originally defined on the $(d-1)$-dimensional edge/face  $e_k$, to the entire $d$-dimensional element $T$, we use the extension
$$
v_b (X)= v_b(Proj_{e_k} (X)),
$$
where $X$ is an arbitrary point in $T$, $Proj_{e_k} (X)$ represents the orthogonal projection of $X$ onto  the hyperplane $H\subset\mathbb R^d$  containing   $e_k$. 
If  $Proj_{e_k} (X)$ lies outside $e_k$, then $v_b(Proj_{e_k} (X))$ is defined as the natural extension of  
 $v_b$ from $e_k$ to $H$.

We assert that $v_b$ remains  a polynomial defined on the element $T$ following the extension. 

Consider the hyperplane $H$ that contains the edge/face   $e_k$, which is determined by $d-1$ linearly independent vectors $\beta_1, \cdots, \beta_{d-1}$ originating from a point $A$ on $e_k$. Any point $P$ on  $e_k$ can be parametrized as
$$
P(t_1, \cdots, t_{d-1})=A+t_1\beta_1+\cdots+t_{d-1}\beta_{d-1},
$$
where $t_1, \cdots, t_{d-1}$ are parameters.

Since  $v_b(P(t_1, \cdots, t_{d-1}))$ is a polynomial of degree $p$ defined on $e_k$, it can be expressed as
$$
v_b(P(t_1, \cdots, t_{d-1}))=\sum_{|\alpha|\leq p}c_{\alpha}\textbf{t}^{\alpha},
$$
where $\textbf{t}^{\alpha}=t_1^{\alpha_1}\cdots t_{d-1}^{\alpha_{d-1}}$ and $\alpha=(\alpha_1, \cdots, \alpha_{d-1})$  is a multi-index.

For any point $X$ in the element $T$, the projection of $X$  onto  the hyperplane $H$   is the point in $H$ that minimizes the distance to  $X$.  This projection, denoted as $Proj_{e_k} (X)$, is an affine transformation given by
$$
Proj_{e_k} (X)=A+\sum_{i=1}^{d-1} t_i(X)\beta_i,
$$
where  $A$ is the origin point on $e_k$,  and $t_i(X)$ are the projection coefficients determined by solving the orthogonality condition 
$$
(X-Proj_{e_k} (X))\cdot \beta_j=0,\qquad \forall j=1, \cdots, d-1.
$$
This yields a system of linear equations in $t_1(X)$, $\cdots$, $t_{d-1}(X)$, which can be explicitly solved, ensuring that
$$
t_i(X) \  \text{is a linear function of} \  X.
$$
Consequently, the projection  $Proj_{e_k} (X)$ is an affine linear function  of $X$.

To extend $v_b$ from $e_k$ to the entire element $T$, we define
$$
v_b(X)=v_b(Proj_{e_k} (X))=\sum_{|\alpha|\leq p}c_{\alpha}\textbf{t}(X)^{\alpha},
$$
where $\textbf{t}(X)^{\alpha}=t_1(X)^{\alpha_1}\cdots t_{d-1}(X)^{\alpha_{d-1}}$. Since each $t_i(X)$ is a linear function of $X$, it follows that  $\textbf{t}(X)^{\alpha}$ is a polynomial in $X=(x_1, \cdots, x_d)$, thereby confirming that  $v_b(X)$ remains a polynomial in $d$-dimensional coordinates.

Similarly, let $v_{trace}$ denote the trace of $v_0$ on $e_k$. We extend $v_{trace}$ to $T$  using  the formula 
$$
 v_{trace} (X)= v_{trace}(Proj_{e_k} (X)),
$$
For any $X$ in $T$, if $Proj_{e_k} (X)$ does not lie on $e_k$, then $v_{trace}(Proj_{e_k} (X))$ is defined as the extension of   $v_{trace}$ from $e_k$ to the hyperplane $H$. By a reasoning analogous to that for $v_b$, it follows that $v_{trace}$ remains a polynomial after this extension. 

Define $\varphi=(v_b-v_0)l_k \varphi_{e_k}$. Then,
\begin{equation*}
    \begin{split}
\|\varphi\|^2_T  =
\int_T \varphi^2dT \leq & Ch_T^2\int_T (\nabla\varphi)^2dT
\\ \leq & Ch_T^2  
 \int_T  (\nabla((v_b-v_{trace})(X)l_k  \varphi_{e_k}))^2dT\\
\leq &Ch_T^3 \int_{e_k} ( (v_b-v_{trace})(Proj_{e_k} (X))(\nabla l_k)  \varphi_{e_k})^2ds\\\leq &Ch_T \int_{e_k} (v_b-v_0)^2ds,
    \end{split}
\end{equation*} 
where we applied  Poincare inequality since $\varphi=0$ on  each $e_i$ for $i=1,\cdots,N$,   $\nabla \varphi =0$ on  each $e_i$ for $i \neq k$, $\nabla \varphi =(v_0-v_b)(\nabla l_k) \varphi_{e_k}=\mathcal{O}( \frac{ (v_0-v_b)\varphi_{e_k}}{h_T}\textbf{C})$  on   $e_k$ for some vector constant $\textbf{C}$, along with the properties of the projection.

This completes the proof of the lemma.

\end{proof}

\begin{lemma}\label{phi2}
For any $\{v_0,v_b, v_g\}\in V_h$, define $\varphi=(\kappa \nabla v_0\cdot \textbf{n}-v_{g })  \varphi_{e_k}$. The following inequality holds:
\begin{equation}
  \|\varphi\|_T ^2 \leq Ch_T \int_{e_k}(\kappa \nabla v_0\cdot \textbf{n}-v_{g })^2ds.
\end{equation}
\end{lemma}
\begin{proof}
We begin by extending $v_g$, originally defined on the $(d-1)$-dimensional edge/face $e_k$, to the entire $d$-dimensional polytopal element $T$ using the formula:
$$
 v_g (X)= v_g(Proj_{e_k} (X)),
$$
where $X$ is any point in $T$, and $Proj_{e_k} (X)$ denotes the orthogonal projection of the point $X$ onto the hyperplane $H$ containing     $e_k$. If $Proj_{e_k} (X)$ does not lie on $e_k$, we define 
 $v_g(Proj_{e_k} (X))$ as the extension of $v_g$ from $e_k$ to $H$.

We assert that after this extension, $v_g$ remains a polynomial on $T$, which follows from the same reasoning as in Lemma \ref{phi}.

Next, let  $v_{trace}$ be the trace of $v_0$  on    $e_k$. We extend $v_{trace}$   to   $T$  using the formula:
$$
 v_{trace} (X)= v_{trace}(Proj_{e_k} (X)).
$$
Again, if $Proj_{e_k} (X)$ is not on  $e_k$, we define $v_{trace}(Proj_{e_k} (X))$ as the extension of $v_{trace}$ from $e_k$ to   $H$. It follows from Lemma \ref{phi} that 
  $v_{trace}$ remains a polynomial after this extension.  

Now, setting $\varphi=(\kappa \nabla v_0\cdot \textbf{n}-v_{g })  \varphi_{e_k}$, we obtain:
\begin{equation*}
    \begin{split}
\|\varphi\|^2_T  =
\int_T \varphi^2dT =  &\int_T ((\kappa \nabla v_0\cdot \textbf{n}-v_{g })(X)  \varphi_{e_k})^2dT\\
\leq &Ch_T \int_{e_k} ((\kappa \nabla v_{trace}\cdot \textbf{n}-v_{g })(Proj_{e_k} (X))   \varphi_{e_k})^2dT\\ 
\\\leq &Ch_T \int_{e_k} (\kappa \nabla v_0\cdot \textbf{n}-v_{g })^2ds,
    \end{split}
\end{equation*} 
where we used the facts that (1) $\varphi_{e_k}=0$ on each  edge/face  $e_i$ for $i \neq k$,  (2) there exists a subdomain $\widehat{e_k}\subset e_k$ where $\varphi_{e_k} \geq \rho_1$ for some constant $\rho_1>0$, 
and applied the properties of the projection.

 This completes the proof.
\end{proof}

\begin{lemma}\label{normeqva} There exist two positive constants,  $C_1$ and $C_2$, such that for every $v=\{v_0, v_b, v_g\} \in V_h$, the following norm equivalence holds:
 \begin{equation}\label{normeq}
 C_1\|v\|_{2, h}\leq \3bar v\3bar  \leq C_2\|v\|_{2, h}.
\end{equation}
\end{lemma} 

\begin{proof}    We begin by recalling that the edge/face-based bubble function is given by
$$\varphi_{e_k}= \Pi_{i=1, \cdots, N, i\neq k}l_i^2(x).$$

First, we extend the function $v_b$, originally defined on the edge/face $e_k$, to the entire element $T$. 
Similarly, let  $v_{trace}$ denote the trace of $v_0$ on  $e_k$; we extend $v_{trace}$ to  $T$ as well.  For notational convenience, these extensions are still denoted by  $v_b$ and $v_0$. (Details of these extensions appear in Lemma \ref{phi}.)
By substituting $$\varphi=(v_b-v_0)l_k \varphi_{e_k}$$ into   equation \eqref{A.002}, we obtain 
\begin{equation} \label{t33}
\begin{split}
 & (E_{w}v,
\varphi)_T\\=&(Ev_0,\varphi)_T+\langle
v_0-v_b, \kappa\nabla\varphi\cdot \textbf{n} \rangle_{\partial T}-\langle
 \kappa \nabla v_0\cdot \textbf{n}-v_{g },\varphi  \rangle_{\partial T} \\
 =&(Ev_0,\varphi)_T+    \int_{e_k} |v_b-  v_0|^2   \kappa (\nabla l_k) \varphi_{e_k}\cdot \bn ds\\
  =&(Ev_0,\varphi)_T+ Ch_T^{-1}   \int_{e_k} |v_b-  v_0|^2   \kappa  \varphi_{e_k}  ds,
\end{split}
\end{equation}
   where we have used that $\varphi=0$ on each edge/face $e_i$ for $i=1, \cdots, N$, that $\nabla \varphi   =0$ on each edge/face $e_i$ with $i \neq k$ and that   on   $e_k$, the gradient satisfies $$\nabla \varphi =(v_0-v_b)(\nabla l_k) \varphi_{e_k}=\mathcal{O}( \frac{ (v_0-v_b)\varphi_{e_k}}{h_T}\textbf{C})$$  with $\textbf{C}$ a constant vector.

 Recall additionally that: (1) $\varphi_{e_k}=0$ on every edge/face  $e_i$ for $i \neq k$, (2) there exists a subdomain $\widehat{e_k}\subset e_k$ such that  $\varphi_{e_k} \geq \rho_1$ for some constant $\rho_1>0$.
By invoking the Cauchy–Schwarz inequality and a domain inverse inequality (see \cite{wy3655}), together with \eqref{t33} and the result of Lemma \ref{phi}, we deduce
\begin{equation*}
\begin{split}
  \int_{e_k}|v_b-  v_0|^2  ds\leq & C \int_{e_k} |v_b-  v_0|^2\kappa\varphi_{e_k}   ds 
  \\ \leq& C h_T (\|E_{w} v\|_T+\|E v_0\|_T){ \| \varphi\|_T}\\
 \leq & C { h_T^{\frac{3}{2}}} (\|E_{w} v\|_T+\|E v_0\|_T){ (\int_{e_k}|v_b- v_0|^2ds)^{\frac{1}{2}}},
 \end{split}
\end{equation*}
which, in conjunction with Lemma \ref{norm1}, implies that
\begin{equation}\label{t21}
 h_T^{-3}\int_{e_k}|v_b-  v_0|^2  ds \leq C  (\|E_{w} v\|^2_T+\|E v_0\|^2_T)\leq C\|E_{w} v\|^2_T.   
\end{equation}

Next, we extend $v_g$ from $e_k$ to $T$;  
 for simplicity, the extension is still denoted by  $v_g$ (see Lemma \ref{phi2} for details). 
Letting $\varphi=(\kappa \nabla v_0\cdot \textbf{n}-v_{g })  \varphi_{e_k}$ in \eqref{A.002} yields
\begin{equation*} 
\begin{split}
    &(E_{w}v,
\varphi)_T\\=&(Ev_0,\varphi)_T+\langle
v_0-v_b, \kappa\nabla\varphi\cdot \textbf{n} \rangle_{\partial T}-\langle
 \kappa \nabla v_0\cdot \textbf{n}-v_{g },\varphi  \rangle_{\partial T}\\=&(Ev_0,\varphi)_T+\langle
v_0-v_b, \kappa\nabla\varphi\cdot \textbf{n} \rangle_{\partial T}-\int_{e_k} ( \kappa \nabla v_0\cdot \textbf{n}-v_{g } )^2   \varphi_{e_k}ds
\end{split}
\end{equation*}
where we have used that  $\varphi_{e_k} =0$ on every edge/face $e_i$ with  $i \neq k$, 
and the fact that there exists a subdomain $\widehat{e_k}\subset e_k$ for which  $\varphi_{e_k} \geq \rho_1>0$.
By applying the Cauchy–Schwarz inequality, the domain inverse inequality (see \cite{wy3655}), the inverse inequality, and the trace inequality \eqref{x}, together with \eqref{t21} and Lemma \ref{phi2}, we obtain
 \begin{equation*} 
\begin{split}
&  \int_{e_k}( \kappa \nabla v_0\cdot \textbf{n}-v_{g } )^2  ds\\\leq &C   \int_{e_k}( \kappa \nabla v_0\cdot \textbf{n}-v_{g } )^2   \varphi_{e_k}ds \\
  \leq & C (\|E_{w} v\|_T+\|E v_0\|_T)\| \varphi\|_T+  C\|v_0-v_b\|_{\partial T}\|\kappa \nabla \varphi\cdot\bn\|_{\partial T}\\
 \leq & C h_T^{\frac{1}{2}} (\|E_{w} v\|_T+\|E v_0\|_T)(\int_{e_k} (\kappa \nabla v_0\cdot \textbf{n}-v_{g } )^2ds)^{\frac{1}{2}}\\& + C h_T^{\frac{3}{2}}  \|E_{w} v\|_T  h_T^{-\frac{1}{2}}h_T^{-1}h_T^{\frac{1}{2}}(\int_{e_k}(\kappa \nabla v_0\cdot \textbf{n}-v_{g } )^2ds)^{\frac{1}{2}}. 
 \end{split}
\end{equation*} 
This, together with Lemma \ref{norm1}, gives 
\begin{equation} \label{t11}
 h_T^{-1}\int_{e_k}( \kappa \nabla v_0\cdot \textbf{n}-v_{g } )^2  ds \leq C  (\|E_{w} v\|^2_T+\|E v_0\|^2_T)\leq C\|E_{w} v\|^2_T.
\end{equation}
  By combining Lemmas \ref{norm1}–\ref{norm1-1} with estimates \eqref{t21} and \eqref{t11}, and by employing the norm definitions \eqref{3barnorm} and \eqref{disnorm}, we deduce that
  $$
 C_1\|v\|_{2, h}\leq \3bar v\3bar.
$$
 
For the reverse inequality, we apply Cauchy-Schwarz inequality, the inverse inequality, and  the trace inequality \eqref{x} to  \eqref{A.002}   to obtain
\begin{equation*}
    \begin{split}
  \Big| (E_{w}v, \varphi)_T\Big| \leq &\|E v_0\|_T \|  \varphi\|_T+
 \|v_b-v_0\|_{\partial T} \| \kappa\nabla\varphi \cdot\bn \|_{\partial T}+\|  \kappa \nabla v_0\cdot \textbf{n}-v_{g }  \|_{\partial T} \|\varphi \|_{\partial T} \\
 \leq &\|E v_0\|_T \|  \varphi\|_T+
 h_T^{-\frac{3}{2}}\|v_b-v_0\|_{\partial T} \|  \varphi\|_{  T}+h_T^{-\frac{1}{2}}\|  \kappa \nabla v_0\cdot \textbf{n}-v_{g }  \|_{\partial T} \|\varphi  \|_{T}.
    \end{split}
\end{equation*}
This implies that
$$
\| E_{w}v\|_T^2\leq C( \|E v_0\|^2_T  +
 h_T^{-3}\|v_b-v_0\|^2_{\partial T}+h_T^{-1}\| \kappa \nabla v_0\cdot \textbf{n}-v_{g }\|^2_{\partial T}),
$$
which, combined with Lemma \ref{norm1-1}, gives 
$$ \3bar v\3bar  \leq C_2\|v\|_{2, h}.$$

The two inequalities together establish the desired norm equivalence, thereby completing the proof. 
 \end{proof}

\begin{theorem}
 The WG Algorithm \ref{PDWG1} admits a unique solution. 
\end{theorem}
\begin{proof}
Suppose   $u_h^{(1)}\in V_h$ and $u_h^{(2)}\in V_h$ are two distinct solutions of the WG Algorithm \ref{PDWG1}. Define their difference as  $\eta_h= u_h^{(1)}-u_h^{(2)}$. Then,  $\eta_h\in V_h^0$ satisfies 
$$\sum_{T\in {\cal T}_h}(E_{w} \eta_h,E_{w}v )_T +2\mu (\kappa \nabla_w \eta_h,\nabla_w v )_T +\mu^2(\eta_0,v_0)_T  =0, \quad\forall v\in V_h^0. $$ 
Choosing $v=\eta_h$ in the above equation yields $\3bar \eta_h\3bar=0$. By the norm equivalence \eqref{normeq}, we obtain $\|\eta_h\|_{2,h}=0$, which implies that on each element $T$, 
$$E \eta_0=0,  \quad\nabla \eta_0=0, \quad\eta_0=0,$$ along with the conditions $\eta_0=\eta_b$ and $\kappa \nabla \eta_0 \cdot\bn=\eta_g$ on  $\partial T$. 
Consequently,  $\eta_0$ must be a constant within each element  $T$, which, together with the fact that $ \eta_0=\eta_b$ on each
$\partial T$,   indicates that $\eta_0$ is a constant across the entire domain $\Omega$.  Furthermore, since $ \eta_0=\eta_b$ on
$\partial T$ and $\eta_b|_{\partial T\cap \partial\Omega}=0$,  we conclude that $\eta_0=0$ in $\Omega$. This, together with $ \eta_0=\eta_b$ on 
$\partial T$, implies $\eta_b=0$ in $\Omega$. Similarly, using the boundary condition $ \kappa \nabla
\eta_0 \cdot \textbf{n}= \eta_g$  on  each $\partial T$ and  $\eta_g|_{\partial T\cap \partial\Omega}=0$,  we deduce that   $\eta_g=0$ in $\Omega$. Therefore, $\eta_h=0$ in $\Omega$, which implies $u_h^{(1)}\equiv u_h^{(2)}$.  

This establishes the uniqueness of the solution, completing the proof.
\end{proof}

\section{Technical Results} 
This section aims to establish essential technical results related to $L^2$ projections, which play a crucial role in the error analysis of the  WG   method.

 For each element  $T\in {\cal T}_h$, let 
  $Q^{r_1}$ and  $Q^{r_2}$  be the $L^2$ projection operators onto the finite element spaces consisting of piecewise polynomials of degree at most  
 $r_1$ and $r_2$, respectively.

\begin{lemma}\label{Lemma5.1}   The following properties hold:
\begin{equation}\label{pro1}
E_{w}v =Q^{r_1}(E v), \qquad \forall v\in H_{\kappa}^2(T),
\end{equation}
\begin{equation}\label{pro2}
\nabla_{w}v =Q^{r_2}(\nabla v), \qquad \forall v\in H^1(T).
\end{equation}\end{lemma}

\begin{proof} 
 For any $u\in H_{\kappa}^2(T)$, applying equation \eqref{A.002} yields
\begin{equation*} 
\begin{split}
 & (E_{w}v,
\varphi)_T\\ =&(Ev,\varphi)_T+\langle
v|_T-v|_{\partial T}, \kappa\nabla\varphi\cdot \textbf{n} \rangle_{\partial T}-\langle
 \kappa \nabla v|_T\cdot \textbf{n}-( \kappa \nabla v \cdot \textbf{n})|_{\partial T},\varphi  \rangle_{\partial T}\\
  =&(Ev,\varphi)_T=(Q^{r_1} Ev, \varphi)_T,
\end{split}
\end{equation*}
for all $\varphi \in
 P_{r_1}(T)$.  

Similarly, for any $u\in H^1(T)$,  using equation \eqref{2.4-3}, we obtain
 \begin{equation*} 
  ( \nabla_{ w}v,\boldsymbol{ \psi})_T= (\nabla v,  \boldsymbol{ \psi})_T-
 \langle v|_T-v|_{\partial T},\boldsymbol{ \psi}\cdot  \textbf{n}\rangle_{\partial T}= (\nabla v,  \boldsymbol{ \psi})_T=(Q^{r_2}\nabla v,  \boldsymbol{ \psi})_T
 \end{equation*}
for all $\boldsymbol{ \psi}\in
 [P_{r_2}(T)]^d$.

 This  completes the proof of this lemma.
  \end{proof}

\begin{lemma}\label{Lemma5.2}\cite{wangft, wy3655}  Let ${\cal T}_h$ be a finite element partition of $\Omega$ satisfying the shape regularity assumption as defined in \cite{wy3655}. Then, for any $0\leq s\leq 2$, $0\leq m\leq k$, $1\leq n\leq r_1$, and $0\leq q\leq r_2$, there exists a constant $C$ such that the following estimates hold:
\begin{equation}\label{3.2}
\sum_{T\in {\cal T}_h}h_T^{2s}\|u-Q_0u\|^2_{s,T}\leq Ch^{2(m+1)}\|u\|_{m+1}^2,
\end{equation}
\begin{equation}\label{3.3}
\sum_{T\in {\cal
T}_h}h_T^{2s}\|Eu-Q^{r_1}Eu\|^2_{s,T}\leq Ch^{2(n-1)}\|u\|_{n+1}^2,
\end{equation}
\begin{equation}\label{3.2-2}
\sum_{T\in {\cal
T}_h} h_T^{2s} \| \kappa\nabla u -   Q^{r_2}(\kappa  \nabla u )\|^2_{s,T}\leq Ch^{2q}\|u\|_{q+1}^2.
\end{equation}
\end{lemma}

\begin{lemma}\label{Lemma5.3} \cite{wangft} Let $0\leq m\leq k$, $1\leq n\leq r_1$, $0\leq q\leq r_2$, and $u\in
H^{\max\{n+1,4\}}(\Omega)$. There exists a constant $C$ such that
the following estimates hold true:
\begin{equation}\label{3.5}
\Big(\sum_{T\in {\cal
T}_h} h_T\|Eu-Q^{r_1}(Eu)\|_{\partial T}^2\Big)^{\frac{1}{2}}\leq
Ch^{n-1}\|u\|_{n+1},
\end{equation}
\begin{equation}\label{3.6} 
\Big(\sum_{T\in {\cal
T}_h} h_T^3\|\kappa \nabla(E
 u-Q^{r_1}(Eu)) \cdot \textbf{n}\|_{\partial T}^2\Big)^{\frac{1}{2}}\\
\leq Ch^{n-1}(\|u\|_{n+1}+h\delta_{r_1, 0}\|u\|_4), 
\end{equation}
\begin{equation}\label{3.7}
\Big(\sum_{T\in {\cal T}_h}h_T^{-1}\| \kappa\nabla (Q_0u)\cdot \textbf{n}-Q_g(\kappa \nabla
u\cdot \textbf{n})\|_{\partial T}^2\Big)^{\frac{1}{2}}\leq Ch^{m-1}\|u\|_{m+1},
\end{equation}
\begin{equation}\label{3.8}
\Big(\sum_{T\in {\cal T}_h}h_T^{-3}\| 
  Q_0u - Q_bu\|_{\partial T}^2\Big)^{\frac{1}{2}}\leq
 Ch^{m-1}\|u\|_{m+1},
\end{equation}
\begin{equation}\label{3.8-2}
\Big(\sum_{T\in {\cal T}_h} h_T^3\|(\kappa\nabla u -  Q^{r_2} (\kappa  \nabla u ))\cdot \bn \|_{\partial T}^2\Big)
^{\frac{1}{2}}\leq
 Ch^{q+1}\|u\|_{q+1}.
\end{equation}
Here $\delta_{i,j}$ is the usual Kronecker's delta with value $1$
when $i=j$ and value $0$ otherwise.
\end{lemma}

\begin{proof} To prove (\ref{3.5}), by the trace inequality (\ref{trace-inequality}) and the
estimate (\ref{3.3}), we get
\begin{equation*}
\begin{split}
&\sum_{T\in {\cal T}_h} h_T\|Eu-Q^{r_1}(Eu)\|_{\partial T}^2\\
\leq & C\sum_{T\in {\cal T}_h}  \|Eu-Q^{r_1}(Eu)\|_{T}^2+h_T^2\|Eu-Q^{r_1}(Eu)\|_{1,T}^2 \\
\leq & Ch^{2n-2}\|u\|^2_{n+1}.
\end{split}
\end{equation*}

As to (\ref{3.6}), by the trace inequality (\ref{trace-inequality})
and the estimate (\ref{3.3}), we obtain
\begin{equation*}
\begin{split}
&\sum_{T\in {\cal
T}_h} h_T^3\|\kappa \nabla(E
 u-Q^{r_1}(Eu)) \cdot \textbf{n}\|_{\partial T}^2\\
 \leq & C\sum_{T\in {\cal
T}_h} h_T^2\|\nabla  (Eu-Q^{r_1}(Eu))\|_{T}^2
+h_T^4\| \nabla (Eu-Q^{r_1}(Eu))\|_{1,T}^2\Big)\\
\leq & Ch^{2n-2}\big(\|u\|^2_{n+1}+h^2\delta_{r_1,0}\|u\|_4^2).
\end{split}
\end{equation*}

As to (\ref{3.7}), by  the trace inequality (\ref{trace-inequality})
and the estimate (\ref{3.2}), we have

\begin{equation*}
\begin{split}
& \sum_{T\in {\cal T}_h}h_T^{-1}\| \kappa\nabla (Q_0u)\cdot \textbf{n}-Q_g(\kappa \nabla
u\cdot \textbf{n})\|_{\partial T}^2 \\
\leq& \sum_{T\in {\cal T}_h}h_T^{-1}\| \kappa\nabla (Q_0u)\cdot \textbf{n}- \kappa \nabla
u\cdot \textbf{n} \|_{\partial T}^2 \\
\leq& C\sum_{T\in {\cal T}_h}h_T^{-1}\|  \nabla (Q_0u)  -  \nabla
u  \|_{\partial T}^2 \\
\leq& C\sum_{T\in {\cal T}_h} h_T^{-2}\|\nabla Q_0u-\nabla u\|_{ T}^2+\|\nabla Q_0u-\nabla u\|_{1,T}^2 \\
\leq&  Ch^{2m-2}\|u\|^2_{m+1}.
\end{split}
\end{equation*}

As to (\ref{3.8}), by the trace inequality
(\ref{trace-inequality}) and the estimate (\ref{3.2}), we have
\begin{equation*}
\begin{split}
&\sum_{T\in {\cal T}_h}h_T^{-3}\| Q_0u - Q_bu\|_{\partial T}^2\\
\leq & \sum_{T\in {\cal T}_h}h_T^{-3}\| Q_0u -  u\|_{\partial T}^2\\
 \leq& C\sum_{T\in {\cal T}_h} h_T^{-4}\|Q_0u-u\|_{T}^2+h_T^{-2}\|\nabla(Q_0u-u)\|_{T}^2 \\
\leq&  Ch^{2m-2}\|u\|^2_{m+1}.
\end{split}
\end{equation*}

Finally, as to  (\ref{3.8-2}), by the trace inequality
(\ref{trace-inequality}) and the estimate (\ref{3.2-2}), we have
\begin{equation*}
\begin{split}
& \sum_{T\in {\cal T}_h} h_T^3\|(\kappa\nabla u -    Q^{r_2} (\kappa  \nabla u ))\cdot \bn \|_{\partial T}^2 \\
\leq & \sum_{T\in {\cal T}_h} h_T^2\| \kappa\nabla u -   Q^{r_2} (\kappa  \nabla u )  \|^2_{  T}+h_T^4 \| \kappa\nabla u -    Q^{r_2}(\kappa  \nabla u )  \|_{ 1, T}^2
  \\
\leq&  Ch^{2q+2}\|u\|^2_{q+1}.
\end{split}
\end{equation*}
This completes the proof of the lemma.
\end{proof}

\section{An Error Equation} 
Let $u$ be the exact solution of (\ref{0.1}), and let 
 $u_h=\{u_0,u_b, u_g\} \in V_h$ be the weak Galerkin finite element approximation satisfying (\ref{PDWG1}). Define the error function as 
\begin{equation}\label{error-term}
e_h=u-u_h.
\end{equation} 
The objective of this section is to derive
an error equation for $e_h$.

\begin{lemma}\label{Lemma6.1} The error function
$e_h\in V_h^0$ as defined by (\ref{error-term})   satisfies the following equation
\begin{equation}\label{4.1}
\sum_{T\in {\cal T}_h}(E_{w} e_h,E_{w}v )_T +2\mu (\kappa \nabla_w e_h,\nabla_w v )_T +\mu^2(e_0,v_0)_T =\phi_u(v),\qquad
\forall v\in V_h^0,
\end{equation}
where
\begin{equation}\label{phiu}
\begin{split}
\phi_u(v)=
&\sum_{T\in {\cal T}_h}\Big( -\langle  \kappa  \nabla (Eu-  
Q^{r_1}(Eu)  ) \cdot \textbf{n},
 v_0 -v_b \rangle_{\partial T} \\
&+ \langle  \kappa \nabla v_0\cdot \textbf{n}-v_g, Eu-Q^{r_1}E u\rangle_{\partial
T}\\
&+  2\mu  \langle v_0-v_b,   \kappa(\nabla u -     Q^{r_2}\nabla u)\cdot \bn   \rangle_{\partial T}\Big).\end{split}
\end{equation}
\end{lemma}

\begin{proof} Using (\ref{A.002}) with
$\varphi=E_w u$, from (\ref{pro1}),  we obtain
\begin{equation*}
\begin{split}
(E_wv,E_{ w}u)_T=&(E_wv, Q^{r_1}Eu)_T\\=&(Ev_0, Q^{r_1}Eu)_T+\langle v_0-v_b, \kappa\nabla (Q^{r_1}Eu)
\cdot \textbf{n} \rangle_{\partial T}\\
&-\langle\kappa \nabla v_0\cdot \textbf{n}-v_{g},Q^{r_1}Eu\rangle_{\partial T}\\
=&(Ev_0, Eu)_T+\langle
v_0-v_b,\kappa \nabla (Q^{r_1}Eu)\cdot
\textbf{n} \rangle_{\partial T}\\
&-\langle\kappa \nabla v_0\cdot
\textbf{n}  -v_{g },Q^{r_1}E
u\rangle_{\partial T},
\end{split}
\end{equation*}
which implies that
\begin{equation}\label{4.2}
\begin{split}
(Ev_0, Eu)_T=&
(E_{ w}u,E_{ w}v)_T-\langle v_0-v_b,
\kappa \nabla(Q^{r_1}Eu)\cdot \textbf{n} \rangle_{\partial T}\\
&+\langle \kappa \nabla v_0\cdot
\textbf{n}-v_{g },Q^{r_1}E
u\rangle_{\partial T}.
\end{split}
\end{equation}
 
 Next, it follows from the
integration by parts that
\begin{equation}\label{part1}
 \begin{split}
\sum_{T\in {\cal T}_h}(Eu, Ev_0)_T=&
\sum_{T\in {\cal T}_h} (E^2u,v_0)_T-\langle  \kappa \nabla (Eu) \cdot \textbf{n},
 v_0 \rangle_{\partial T}
+\langle  \kappa \nabla v_0\cdot \textbf{n}, Eu\rangle_{\partial
T} .
 \end{split}
\end{equation}

  Using  (\ref{2.4-3}) with  $\boldsymbol{\psi}=\kappa Q^{r_2}(\nabla u)$, from  (\ref{pro2})  and  the  integration by parts, we have
\begin{equation}\label{part2}
 \begin{split}
& \sum_{T\in {\cal T}_h} 2\mu (\kappa \nabla_wu, \nabla_w v)_T\\
=& \sum_{T\in {\cal T}_h} 2\mu (\kappa Q^{r_2}(\nabla u), \nabla_w v)_T\\
 =&  \sum_{T\in {\cal T}_h}2\mu ( \nabla v_0,\kappa Q^{r_2}(\nabla u))_T -2\mu \langle v_0-v_b ,  \kappa Q^{r_2}(\nabla u) \cdot \bn\rangle_{\partial T}\\
 =&  \sum_{T\in {\cal T}_h}2\mu (\nabla v_0,\kappa   \nabla u)_T -2\mu  \langle v_0-v_b , \kappa Q^{r_2}(\nabla u) \cdot \bn\rangle_{\partial T}\\
=&  \sum_{T\in {\cal T}_h}-2\mu (v_0, \nabla \cdot( \kappa \nabla u ))_T + 2\mu  \langle  v_0,  \kappa \nabla u \cdot  \bn  \rangle_{\partial T} \\&-2\mu \langle v_0-v_b , \kappa Q^{r_2}(\nabla u) \cdot \bn\rangle_{\partial T}\\
=& \sum_{T\in {\cal T}_h} -2\mu (v_0, E u)_T + 2\mu  \langle v_0-v_b, \kappa(\nabla u -      Q^{r_2}(\nabla u))\cdot \bn   \rangle_{\partial T}  ,
 \end{split}
\end{equation}
where we have used the fact that the sum for the terms associated
with $v_{b}$   vanishes  (note that 
$v_{b}$  vanishes on $\partial T\cap\partial\Omega$).

 Testing $v_0$ on both sides of the first equation of \eqref{0.1} gives  
\begin{equation}\label{ee}
    \sum_{T\in {\cal T}_h} (E^2
u,v_0)_T-2\mu (Eu,v_0)_T+\mu^2(u,v_0)_T = \sum_{T\in {\cal T}_h}(f,v_0)_T. 
\end{equation}
  Adding (\ref{part1})-(\ref{part2}) and using \eqref{ee}, we have 
\begin{equation*}
\begin{split}
&\sum_{T\in {\cal T}_h} (Eu, Ev_0)_T+ 2\mu \kappa (\nabla_w  u, \nabla_w v)_T+\mu^2 (u,v_0)_T  \\
=& (f,v_0)+\sum_{T\in {\cal T}_h}\Big( -\langle  \kappa \nabla (Eu) \cdot \textbf{n},
 v_0 -v_b \rangle_{\partial T} +  \langle  \kappa \nabla v_0\cdot \textbf{n}-v_g, Eu\rangle_{\partial
T}\\
&+  2\mu  \langle v_0-v_b,    \kappa(\nabla u -      Q ^{r_2}    \nabla u )\cdot \bn   \rangle_{\partial T}\Big),
\end{split}
\end{equation*}
where we have used the fact that the sum for the terms associated
with $v_{b}$ and $v_g  $ vanishes  (note that both
$v_{b}$ and $v_g$ vanish on $\partial T\cap\partial\Omega$). Combining the
above equation with (\ref{4.2})    yield 
\begin{equation*}\label{4.3}
\begin{split}
&\sum_{T\in {\cal T}_h}(E_{ w}  u,E_{ w}v)_T+ 2\mu \kappa (\nabla_w u, \nabla_w v)_T+\mu^2 (u,v_0)_T \\  
=& (f,v_0)+\sum_{T\in {\cal T}_h}\Big( -\langle  \kappa  \nabla (Eu-  
Q^{r_1}(Eu)  ) \cdot \textbf{n},
 v_0 -v_b \rangle_{\partial T} \\
&+ \langle  \kappa \nabla v_0\cdot \textbf{n}-v_g, Eu-Q^{r_1}E u\rangle_{\partial
T}\\
&+  2\mu  \langle v_0-v_b,   \kappa(\nabla u -     Q^{r_2}\nabla u)\cdot \bn   \rangle_{\partial T}\Big),
\end{split}
\end{equation*} 
which, subtracting (\ref{PDWG1}),   completes the proof.
\end{proof}

\section{Error Estimates} 
This section derives error estimates for the solution obtained using the stabilizer-free weak Galerkin algorithm.

\begin{theorem}  Let $u_h\in V_h$ be the weak Galerkin finite element solution resulting from \eqref{PDWG1}, using finite elements of order $k\geq 2$. Suppose the exact solution $u$ of \eqref{0.1} is sufficiently regular, satisfying $u\in H^{ k+1 }(\Omega)$. Then, there exists a constant $C$ such that
\begin{equation}\label{4}
\3bar u-Q_hu\3bar\leq
Ch^{k-1} \|u\|_{k+1}.
\end{equation} 
\end{theorem}
\begin{proof}
Using \eqref{A.002}, the trace inequality \eqref{x}, Cauchy-Schwartz inequality, the inverse inequality, the estimates \eqref{3.2}, \eqref{3.7} and \eqref{3.8} we derive: 
\begin{equation*}
    \begin{split}
&\sum_{T\in {\cal T}_h} (E_w (u-Q_hu), \phi)_T    \\=& 
\sum_{T\in {\cal T}_h}(E(u-Q_0u),\phi)_T+\langle
Q_bu-Q_0u, \kappa\nabla\phi\cdot \textbf{n} \rangle_{\partial T}\\
&-\langle
 \kappa \nabla (u-Q_0u)\cdot \textbf{n}-(u_g-Q_{g}(\kappa \nabla  u \cdot\bn)),\phi  \rangle_{\partial T}\\
  \\=& 
\sum_{T\in {\cal T}_h}(E(u-Q_0u),\phi)_T+\langle
Q_bu-Q_0u, \kappa\nabla\phi\cdot \textbf{n} \rangle_{\partial T}\\
&-\langle -
 \kappa \nabla  Q_0u \cdot \textbf{n}+Q_{g}(\kappa \nabla  u \cdot\bn),\phi  \rangle_{\partial T}\\
 \leq & (\sum_{T\in {\cal T}_h}\|E(u-Q_0u)\|_T^2)^{\frac{1}{2}} ( \sum_{T\in {\cal T}_h}\|\phi\|_T^2)^{\frac{1}{2}}
\\ &+ (\sum_{T\in {\cal T}_h}\|Q_bu-Q_0u\|_{\partial T}^2)^{\frac{1}{2}} ( \sum_{T\in {\cal T}_h}\|\kappa\nabla\phi\cdot \textbf{n} \|_{\partial T}^2)^{\frac{1}{2}}\\
& + (\sum_{T\in {\cal T}_h}\|-
 \kappa \nabla  Q_0u \cdot \textbf{n}+Q_{g}(\kappa \nabla  u \cdot\bn)\|_{\partial T}^2)^{\frac{1}{2}} ( \sum_{T\in {\cal T}_h}\|\phi\|_{\partial T}^2)^{\frac{1}{2}}\\
 \leq & 
Ch^{k-1} \|u\|_{k+1}( \sum_{T\in {\cal T}_h}\|\phi\|_T^2)^{\frac{1}{2}}.
    \end{split}
\end{equation*}
Setting $\phi=E_w (u-Q_hu)$ leads to
\begin{equation}\label{q1}
    \sum_{T\in {\cal T}_h} \|E_w (u-Q_hu)\|^2_T \leq Ch^{2k-2} \|u\|^2_{k+1} 
\end{equation}
 
Using \eqref{2.4-3}, the trace inequality \eqref{x}, Cauchy-Schwartz inequality, \eqref{3.2}, \eqref{3.8}, we have
    \begin{equation*}
    \begin{split}
&\sum_{T\in {\cal T}_h}  2\mu (\kappa \nabla_w (u-Q_hu), \boldsymbol{ \psi})_T     
\\
= & \sum_{T\in {\cal T}_h}   2\mu\kappa(( \nabla (u-Q_0u),  \boldsymbol{ \psi})_T-
 \langle Q_bu-Q_0u ,\boldsymbol{ \psi}\cdot  \textbf{n}\rangle_{\partial T})\\
\leq & (\sum_{T\in {\cal T}_h} \|   \nabla (u-Q_0u)\|^2_T) ^{\frac{1}{2}}  (\sum_{T\in {\cal T}_h}\|\boldsymbol{ \psi}\|^2_T)^{\frac{1}{2}}+ (\sum_{T\in {\cal T}_h} 
 \| Q_bu-Q_0u\|^2_{\partial T})^{\frac{1}{2}}(\sum_{T\in {\cal T}_h}\|\boldsymbol{ \psi}\cdot  \textbf{n}\|^2_{\partial T})^{\frac{1}{2}}\\
 \leq & 
Ch^{k} \|u\|_{k+1}  ( \sum_{T\in {\cal T}_h}\|\boldsymbol{ \psi}\|_T^2)^{\frac{1}{2}}.\end{split}
\end{equation*}

Letting $\boldsymbol{ \psi}=\nabla_w (u-Q_hu) $ gives
   \begin{equation}\label{q2}
    \begin{split}
  \sum_{T\in {\cal T}_h}  2\mu (\kappa \nabla_w (u-Q_hu), \nabla _w(u-Q_hu))_T    \leq  
Ch^{2k} \|u\|^2_{k+1}. \end{split}
\end{equation}

Using Cauchy-Schwartz inequality, \eqref{3.2}, we have
  \begin{equation}\label{q3}
    \begin{split}
\sum_{T\in {\cal T}_h}  \mu^2(u-Q_0u, u-Q_0u)_T\leq  (\sum_{T\in {\cal T}_h} 
 \| u-Q_0u\|^2_{T}) \leq Ch^{2(k+1)}\|u\|^2_{k+1}.
    \end{split}
\end{equation}

Combining \eqref{q1}, \eqref{q2} and \eqref{q3} completes the proof \eqref{4}.
\end{proof}

\begin{theorem}  Let $k\geq 2$. Let $u_h\in V_h$ be the weak Galerkin finite element solution of (\ref{PDWG1}).  Assume that the exact solution $u$ of (\ref{0.1})
satisfies $u\in H^{\max\{k+1,4\}}(\Omega)$.
Then, there exists a constant $C$ such that
\begin{equation} 
\3bar u- u_h\3bar\leq
Ch^{k-1}\Big(\|u\|_{k+1}+h\delta_{r_1,0}\|u\|_4\Big).
\end{equation} 
\end{theorem}
\begin{proof}  We estimate each term on the right-hand side of (\ref{4.1}) as follows:
For the first term, applying the Cauchy-Schwarz inequality and using estimate (\ref{3.6}), we obtain
\begin{equation*} 
\begin{split}
 &\Big|\sum_{T\in {\cal T}_h}\langle  \kappa \nabla (Eu-Q^{r_1} Eu) \cdot \textbf{n},
v_0-v_b \rangle_{\partial T}\Big|\\
 \leq&\Big(\sum_{T\in {\cal T}_h}h_T^3\| \kappa \nabla (Eu-Q^{r_1}Eu) \cdot \textbf{n}\|^2_{\partial T}\Big)^{\frac{1}{2}}
\Big(\sum_{T\in {\cal T}_h}h_T^{-3}\| v_0-v_b \|^2_{\partial T}\Big)^{\frac{1}{2}}\\
  \leq& Ch^{k-1}(\|u\|_{k+1}+h\delta_{r_1,0}\|u\|_4)\|v\|_{2,h}.
\end{split}
\end{equation*}For the second term, again applying the Cauchy-Schwarz inequality and using estimate (\ref{3.5}), we obtain
\begin{equation*}\label{4.6}
\begin{split}
 &\Big|\sum_{T\in {\cal T}_h}\langle  \kappa \nabla v_0\cdot \textbf{n}-v_g, Eu-Q^{r_1}Eu\rangle_{\partial
T}\Big|\\
 \leq&\Big(\sum_{T\in {\cal T}_h}h_T^{-1}\|\kappa \nabla v_0\cdot \textbf{n}-v_g\|^2_{\partial T}\Big)^{\frac{1}{2}}
\Big(\sum_{T\in {\cal T}_h}h_T \|Eu-Q^{r_1}Eu\|^2_{\partial T}\Big)^{\frac{1}{2}}\\
  \leq& Ch^{k-1} \|u\|_{k+1} \|v\|_{2,h}.
\end{split}
\end{equation*}For the third term, using the Cauchy-Schwarz inequality and estimate (\ref{3.8-2}), we derive
\begin{equation*} 
\begin{split}
 &\sum_{T\in {\cal T}_h} 2\mu  \langle v_0-v_b,     \kappa(\nabla u -     Q^{r_2}     \nabla u )\cdot \bn   \rangle_{\partial T}\\
 \leq&\Big(\sum_{T\in {\cal T}_h}h_T^{-3}\|v_0-v_b\|^2_{\partial T}\Big)^{\frac{1}{2}}
\Big(\sum_{T\in {\cal T}_h}h_T ^3\| \kappa(\nabla u -     Q^{r_2}     \nabla u )\cdot \bn\Big)^{\frac{1}{2}}\\
  \leq& Ch^{k+1} \|u\|_{k+1} \|v\|_{2,h}.
\end{split}
\end{equation*}
 Substituting these estimates into (\ref{4.1}) and using \eqref{normeq}, we obtain
\begin{equation}\label{s1}
\begin{split}
&(E_{w} e_h,E_{w}v ) +2\mu (\kappa \nabla_w e_h,\nabla_w v ) +\mu^2(e_h,v )  \\
\leq &   Ch^{k-1}(\|u\|_{k+1}+h\delta_{r_1,0}\|u\|_4)\|v\|_{2,h}\\
\leq &   Ch^{k-1}(\|u\|_{k+1}+h\delta_{r_1,0}\|u\|_4)\3bar v\3bar.
\end{split}
\end{equation}
 Choosing  $v=Q_u-u_h$ in \eqref{s1} and applying the Cauchy-Schwarz inequality and the estimate \eqref{4}, we derive
\begin{equation*}\label{s2}
\begin{split}
&\3bar e_h\3bar^2\\
=& (E_{w} e_h,E_{w}(u-Q_hu) ) +2\mu (\kappa \nabla_w e_h,\nabla_w (u-Q_hu) ) +\mu^2(e_h,(u-Q_hu) ) \\
&+(E_{w} e_h,E_{w} (Q_hu-u_h) ) +2\mu (\kappa \nabla_w e_h,\nabla_w (Q_hu-u_h) ) +\mu^2(e_h,(Q_hu-u_h) ) \\
\leq & \3bar e_h\3bar \3bar u-Q_hu\3bar+  Ch^{k-1}(\|u\|_{k+1}+h\delta_{r_1,0}\|u\|_4)\3bar Q_hu-u_h\3bar \\
\leq & \3bar e_h\3bar \3bar u-Q_hu\3bar+  Ch^{k-1}(\|u\|_{k+1}+h\delta_{r_1,0}\|u\|_4)(\3bar Q_hu-u\3bar+\3bar u-u_h\3bar) \\
\leq &Ch^{k-1}(\|u\|_{k+1}+h\delta_{r_1,0}\|u\|_4)\3bar e_h\3bar+Ch^{k-1}(\|u\|_{k+1}+h\delta_{r_1,0}\|u\|_4)h^{k-1}\|u\|_{k+1}.
\end{split}
\end{equation*}This establishes the desired bound, concluding the proof. 
\end{proof}

\section{Numerical test}
In the first numerical example,  we solve \eqref{0.1} on the unit square domain $\Omega
  =(0,1)\times(0,1)$ with the following parameters and the exact solution, 
\an{\label{s-1} \ad{ \mu&=1, \quad \kappa=\p{2 & 0 \\ 0 & 2}, \\
                     u&=(x-x^2)^2(y-y^2)^2.  } }

     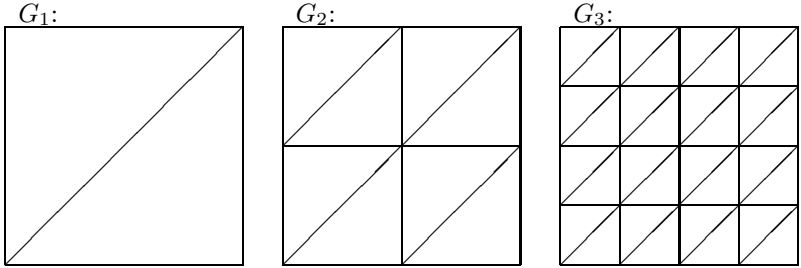
\begin{figure}[H] \setlength\unitlength{1pt}\begin{center}
    \begin{picture}(300,100)(0,0)
     \def\mc{\begin{picture}(90,90)(0,0)
       \put(0,0){\line(1,0){90}} \put(0,90){\line(1,0){90}}
      \put(0,0){\line(0,1){90}}  \put(90,0){\line(0,1){90}}  \put(0,0){\line(1,1){90}}  
      \end{picture}}

    \put(0,0){\mc} \put(5,92){$G_1$:} \put(110,92){$G_2$:}\put(215,92){$G_3$:}
      \put(105,0){\setlength\unitlength{0.5pt}\begin{picture}(90,90)(0,0)
    \put(0,0){\mc}\put(90,0){\mc}\put(0,90){\mc}\put(90,90){\mc}\end{picture}}
      \put(210,0){\setlength\unitlength{0.25pt}\begin{picture}(90,90)(0,0)
    \multiput(0,0)(90,0){4}{\multiput(0,0)(0,90){4}{\mc}} \end{picture}}
    \end{picture}
 \caption{The first three grids for the computation in
    Tables \ref{t-1}--\ref{t-2}. }\label{grid1} 
    \end{center} \end{figure}

We compute the finite element solutions for \eqref{s-1} on uniform 
       triangular grids shown in Figure \ref{grid1} by 
  the  $P_k$/$P_k$/$P_{k-1}$ WG finite elements, defined in \eqref{Wk}, 
      for $k=2,3$ and $4$.
The results are listed in Table \ref{t-1}. 
The optimal orders of convergence are achieved for all solutions in all norms.

\begin{table}[H]
  \centering  \renewcommand{\arraystretch}{1.1}
  \caption{The error and the computed order of convergence 
     for the solution \eqref{s-1} on Figure \ref{grid1} triangular meshes. }
  \label{t-1}
\begin{tabular}{c|cc|cc|cc}
\hline
  $G_i$ &  $\|Q_0  u- u_0\|_0$  & $ h^r $ & $\|\nabla_w(Q_h u- u_h)\|_0$  & $ h^r $ &
      $\|E_w(Q_h u- u_h)\|_0$  & $ h^r $    \\
\hline&\multicolumn{6}{c}{ By the $P_2$/$P_2$/$P_{1}$ WG finite element \eqref{Wk}.}\\
\hline 
 4&     0.183E-3 &  1.5&     0.196E-2 &  1.6&     0.247E+0 &  0.9 \\
 5&     0.502E-4 &  1.9&     0.536E-3 &  1.9&     0.125E+0 &  1.0 \\
 6&     0.128E-4 &  2.0&     0.137E-3 &  2.0&     0.629E-1 &  1.0 \\
\hline&\multicolumn{6}{c}{ By the $P_3$/$P_3$/$P_{2}$ WG finite element \eqref{Wk}.}\\
\hline 
 3&     0.360E-4 &  3.2&     0.660E-3 &  2.9&     0.835E-1 &  1.8 \\
 4&     0.277E-5 &  3.7&     0.808E-4 &  3.0&     0.221E-1 &  1.9 \\
 5&     0.184E-6 &  3.9&     0.982E-5 &  3.0&     0.561E-2 &  2.0 \\
\hline&\multicolumn{6}{c}{ By the $P_4$/$P_4$/$P_{3}$ WG finite element \eqref{Wk}.}\\
\hline 
 2&     0.243E-4 &  3.8&     0.718E-3 &  3.2&     0.669E-1 &  2.4 \\
 3&     0.140E-5 &  4.1&     0.777E-4 &  3.2&     0.126E-1 &  2.4 \\
 4&     0.351E-7 &  5.3&     0.528E-5 &  3.9&     0.174E-2 &  2.9 \\
\hline
    \end{tabular}%
\end{table}%

In the second numerical example,  we solve \eqref{0.1} on the unit square domain $\Omega
  =(0,1)\times(0,1)$ with the following parameters and the exact solution, 
\an{\label{s-2} \ad{ \mu&=1, \quad \kappa=\p{2 & -1 \\ -1 & 2}, \\
                     u&=(x-x^2)^2(y-y^2)^2.  } }

We compute the finite element solutions for \eqref{s-2} on uniform 
       triangular grids shown in Figure \ref{grid1} by 
  the  $P_k$/$P_k$/$P_{k-1}$ WG finite elements, defined in \eqref{Wk}, 
      for $k=2,3$ and $4$.
The results are listed in Table \ref{t-2}. 
The optimal orders of convergence are achieved for all solutions in all norms.

\begin{table}[H]
  \centering  \renewcommand{\arraystretch}{1.1}
  \caption{The error and the computed order of convergence 
     for the solution \eqref{s-2} on Figure \ref{grid1} triangular meshes. }
  \label{t-2}
\begin{tabular}{c|cc|cc|cc}
\hline
  $G_i$ &  $\|Q_0  u- u_0\|_0$  & $ h^r $ & $\|\nabla_w(Q_h u- u_h)\|_0$  & $ h^r $ &
      $\|E_w(Q_h u- u_h)\|_0$  & $ h^r $    \\
\hline&\multicolumn{6}{c}{ By the $P_2$/$P_2$/$P_{1}$ WG finite element \eqref{Wk}.}\\
\hline  
 4&     0.267E-3 &  1.2&     0.258E-2 &  1.3&     0.327E+0 &  0.9 \\
 5&     0.858E-4 &  1.6&     0.821E-3 &  1.7&     0.165E+0 &  1.0 \\
 6&     0.236E-4 &  1.9&     0.226E-3 &  1.9&     0.830E-1 &  1.0 \\
\hline&\multicolumn{6}{c}{ By the $P_3$/$P_3$/$P_{2}$ WG finite element \eqref{Wk}.}\\
\hline 
 3&     0.586E-4 &  2.9&     0.780E-3 &  2.8&     0.106E+0 &  1.8 \\
 4&     0.474E-5 &  3.6&     0.910E-4 &  3.1&     0.282E-1 &  1.9 \\
 5&     0.313E-6 &  3.9&     0.106E-4 &  3.1&     0.716E-2 &  2.0 \\
\hline&\multicolumn{6}{c}{ By the $P_4$/$P_4$/$P_{3}$ WG finite element \eqref{Wk}.}\\
\hline 
 2&     0.291E-4 &  3.8&     0.783E-3 &  3.2&     0.889E-1 &  2.5 \\
 3&     0.216E-5 &  3.8&     0.888E-4 &  3.1&     0.162E-1 &  2.5 \\
 4&     0.482E-7 &  5.5&     0.578E-5 &  3.9&     0.223E-2 &  2.9 \\
\hline
    \end{tabular}%
\end{table}%

     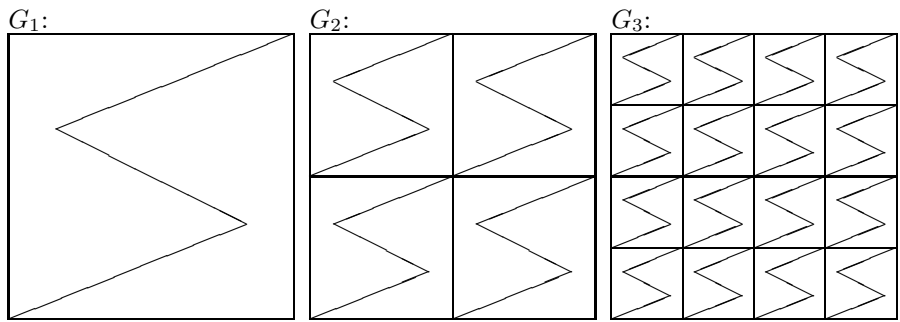
\begin{figure}[H] \setlength\unitlength{1.2pt}\begin{center}
    \begin{picture}(280,98)(0,0)
     \def\mc{\begin{picture}(90,90)(0,0)
       \put(0,0){\line(1,0){90}} \put(0,90){\line(1,0){90}}
      \put(0,0){\line(0,1){90}}  \put(90,0){\line(0,1){90}} 
      
       \put(15,60){\line(2,-1){60}}  
       \put(0,0){\line(5,2){75}}         \put(90,90){\line(-5,-2){75}}  
      \end{picture}}

    \put(0,0){\mc} \put(0,92){$G_1$:} \put(95,92){$G_2$:}  \put(190,92){$G_3$:}
    
      \put(95,0){\setlength\unitlength{0.6pt}\begin{picture}(90,90)(0,0)
    \put(0,0){\mc}\put(90,0){\mc}\put(0,90){\mc}\put(90,90){\mc}\end{picture}}
     \put(190,0){\setlength\unitlength{0.3pt}\begin{picture}(90,90)(0,0)
      \multiput(0,0)(90,0){4}{\multiput(0,0)(0,90){4}{\mc}} \end{picture}}
    \end{picture}
 \caption{The first three nonconvex polygonal grids for the computation in
    Tables \ref{t-3}--\ref{t-4}. }\label{grid2} 
    \end{center} \end{figure}

We compute the two numerical examples again on nonconvex polygonal meshes, shown in
   Figure \ref{grid2}.
We apply  the  $P_k$/$P_k$/$P_{k-1}$ WG finite elements, defined in \eqref{Wk}, 
      for $k=2,3$ and $4$.
The results are listed in  Tables \ref{t-3}--\ref{t-4}. 
We obtain the optimal orders of convergence for all cases.

\begin{table}[H]
  \centering  \renewcommand{\arraystretch}{1.1}
  \caption{The error and the computed order of convergence 
     for the solution \eqref{s-1} on Figure \ref{grid2} polygonal meshes. }
  \label{t-3}
\begin{tabular}{c|cc|cc|cc}
\hline
  $G_i$ &  $\|Q_0  u- u_0\|_0$  & $ h^r $ & $\|\nabla_w(Q_h u- u_h)\|_0$  & $ h^r $ &
      $\|E_w(Q_h u- u_h)\|_0$  & $ h^r $    \\
\hline&\multicolumn{6}{c}{ By the $P_2$/$P_2$/$P_{1}$ WG finite element \eqref{Wk}.}\\
\hline  
 4&     0.152E-3 &  1.6&     0.165E-2 &  1.7&     0.199E+0 &  0.9 \\
 5&     0.405E-4 &  1.9&     0.435E-3 &  1.9&     0.102E+0 &  1.0 \\
 6&     0.102E-4 &  2.0&     0.109E-3 &  2.0&     0.513E-1 &  1.0 \\
\hline&\multicolumn{6}{c}{ By the $P_3$/$P_3$/$P_{2}$ WG finite element \eqref{Wk}.}\\
\hline 
 3&     0.229E-4 &  3.3&     0.800E-3 &  2.5&     0.586E-1 &  2.2 \\
 4&     0.164E-5 &  3.8&     0.109E-3 &  2.9&     0.156E-1 &  1.9 \\
 5&     0.107E-6 &  3.9&     0.138E-4 &  3.0&     0.400E-2 &  2.0 \\
\hline&\multicolumn{6}{c}{ By the $P_4$/$P_4$/$P_{3}$ WG finite element \eqref{Wk}.}\\
\hline 
 2&     0.133E-4 &  5.4&     0.957E-3 &  5.3&     0.801E-1 &  4.8 \\
 3&     0.484E-6 &  4.8&     0.120E-3 &  3.0&     0.870E-2 &  3.2 \\
 4&     0.145E-7 &  5.1&     0.792E-5 &  3.9&     0.118E-2 &  2.9 \\
\hline
    \end{tabular}%
\end{table}%

\begin{table}[H]
  \centering  \renewcommand{\arraystretch}{1.1}
  \caption{The error and the computed order of convergence 
     for the solution \eqref{s-2} on Figure \ref{grid2} polygonal meshes. }
  \label{t-4}
\begin{tabular}{c|cc|cc|cc}
\hline
  $G_i$ &  $\|Q_0  u- u_0\|_0$  & $ h^r $ & $\|\nabla_w(Q_h u- u_h)\|_0$  & $ h^r $ &
      $\|E_w(Q_h u- u_h)\|_0$  & $ h^r $    \\
\hline&\multicolumn{6}{c}{ By the $P_2$/$P_2$/$P_{1}$ WG finite element \eqref{Wk}.}\\
\hline  
 4&     0.227E-3 &  1.3&     0.225E-2 &  1.4&     0.288E+0 &  0.9 \\
 5&     0.676E-4 &  1.7&     0.657E-3 &  1.8&     0.148E+0 &  1.0 \\
 6&     0.178E-4 &  1.9&     0.172E-3 &  1.9&     0.743E-1 &  1.0 \\
\hline&\multicolumn{6}{c}{ By the $P_3$/$P_3$/$P_{2}$ WG finite element \eqref{Wk}.}\\
\hline 
 3&     0.328E-4 &  3.3&     0.113E-2 &  2.0&     0.770E-1 &  2.1 \\
 4&     0.231E-5 &  3.8&     0.158E-3 &  2.8&     0.209E-1 &  1.9 \\
 5&     0.148E-6 &  4.0&     0.197E-4 &  3.0&     0.537E-2 &  2.0 \\
\hline&\multicolumn{6}{c}{ By the $P_4$/$P_4$/$P_{3}$ WG finite element \eqref{Wk}.}\\
\hline 
 2&     0.157E-4 &  5.2&     0.112E-2 &  5.2&     0.104E+0 &  4.8 \\
 3&     0.702E-6 &  4.5&     0.128E-3 &  3.1&     0.112E-1 &  3.2 \\
 4&     0.201E-7 &  5.1&     0.919E-5 &  3.8&     0.152E-2 &  2.9 \\
\hline
    \end{tabular}%
\end{table}%

\end{document}